\documentclass[10pt]{article}
\usepackage{amsfonts,amsmath,amsthm, amssymb}
\usepackage{latexsym, euscript, epic, eepic}
\newtheorem{theorem}{Theorem}[section]
\newtheorem{lemma}[theorem]{Lemma}%[section]
\newtheorem{definition}{Definition}[section]
\newtheorem{proposition}[theorem]{Proposition}%[section]

\newtheorem*{TheoremA}{Theorem A}

\def\diam{\mathop{\hbox{\rm diam}}}
\begin{document}

\title{Variational principles for topological pressures on subsets }
\author{     Xinjia Tang$^{\dag}$, Wen-Chiao Cheng$^\S$ and Yun Zhao$^{\ddag}$ \\
\small\it  Department of Mathematics, Soochow  University, Suzhou 215006, Jiangsu, P.R.China\\
  \small\it ( email: $\dag$ tangxinjia111@126.com, $\ddag$ zhaoyun@suda.edu.cn)\\
 \small \it $\S$ Department of Applied Mathematics, Chinese Culture University\\
\small \it Yangmingshan, Taipei 11114, Taiwan, e-mail: zwq2@faculty.pccu.edu.tw }
\date{}
%\footnotetext {$^{\ddag}$ Corresponding author}
 \footnotetext{2000 {\it Mathematics Subject classification}:
 37D35, 37A35, 37C45}
 \maketitle

\begin{center}
\begin{minipage}{120mm}
{\small {\bf Abstract.}
 The goal of this paper is to define and investigate those topological pressures, which is an extension of
topological  entropy presented by Feng and Huang \cite{fh},
of continuous transformations. This study reveals the similarity between many known results of topological pressure.
More precisely, the investigation of the variational principle is given and related propositions are also described.
That is, this paper defines the measure theoretic pressure $P_{\mu}(T,f)$ for any
$\mu\in{\mathcal M(X)}$, and shows that $
P_B(T,f,K)=\sup\bigr\{P_{\mu}(T,f):{\mu}\in{\mathcal
 M(X)},{\mu}(K)=1\bigr\}$,
where $K\subseteq X$ is a non-empty compact subset and
$P_B(T,f,K)$ is the Bowen topological pressure on $K$.
Furthermore, if $Z\subseteq X$ is an analytic subset, then $
P_B(T,f,Z)=\sup\bigr\{P_B(T,f,K):K\subseteq Z\ \text{is
compact}\bigr\}$. However, this analysis relies on more techniques of ergodic theory and topological dynamics. }
\end{minipage}
\end{center}

\vskip0.5cm

{\small{\bf Key words and phrases.} \  Measure-theoretic pressure,
Variational principle, Borel Probability measure, Topological
pressure.}\vskip0.5cm

%%%%%%%%%%%%%%%%%%%%%%%%%%%%%%%%%%%%%%%%%%%%%%%%
%%%%%%%%%%%%%%%%%%%%%%%%%%%%%%%%%%%%%%%%%%%%%%%%
\section{Introduction.}
%\setcounter{section}{1} \setcounter{equation}{0}
%%%%%%%%%%%%%%%%%%%%%%%%%%%%%%%%%%%%%%%%%%%%%%%%
%%%%%%%%%%%%%%%%%%%%%%%%%%%%%%%%%%%%%%%%%%%%%%%%

Throughout this paper,  $(X,T)$ denotes a \emph{topological dynamical system} (TDS), that is, $X$ is a compact metric space
with a metric $d$, and $T : X\rightarrow X$ is a continuous transformation. Let $\mathcal M(X)$, $\mathcal{M}_T$ and
$\mathcal{E}_T$ denote the sets of all Borel probability measures, $T$-invariant Borel probability measures on
 and $T$-invariant ergodic measures on $X$, respectively. For any $\mu\in \mathcal M_T$, let $h_{\mu}(T)$ denote
 the measure theoretic entropy of $\mu$ with respect to $T$ and let $h_{top}(T)$ denote the topological entropy of
the system $(X,T)$, see \cite{w2} for precise definitions.  It is well-known that
entropies constitute essential invariants in the characterization
of the complexity of a dynamical system. The classical
measure-theoretic entropy for an invariant measure \cite{kol} and
the topological entropy \cite{akm} are introduced. The basic
relation between topological entropy and measure theoretic entropy
is the variational principle, e.g., see \cite{w2}.

Topological pressure is a non-trivial and natural generalization
of topological entropy. Starting from ideas in the statistical
mechanics of lattice systems, Ruelle \cite{rue} introduced
topological pressure of a continuous function for $ \mathbb{Z}^n$
actions on compact spaces and established the variational
principle for topological pressure in this context when the action
is expansive and satisfies the specification property. Later,
Walters \cite{wal} proved the variational principle for a
$\mathbb{Z}^+-$action without these assumptions. Misiurewicz
\cite{mis} gave a elegant proof of the variational principle for
$\mathbb{Z}_{+}^n$ action. See \cite{op82,o85,st80,tem84,tem92}
for the variational principle for amenable group actions and
\cite{chu} for actions of sofic groups. Moreover, Barreira \cite{ba,ba06,ba10}, Cao-Feng-Huang \cite{cfh},
Mummert \cite{mu06}, Zhao-Cheng \cite{zc11,zc13} dealt with variational principle for
topological pressure with nonadditive potentials,  and Huang-Yi
\cite{hy} and Zhang \cite{zhang}, also considered the variational principle for
the  local topological pressure. This paper
conducts research for $\mathbb{Z}$ or $\mathbb{Z}^+$ actions.

 From a viewpoint of dimension theory,
Pesin and Pitskel' \cite{pes2} defined the topological pressure
for noncompact sets which is a generalization of Bowen's
definition of topological entropy for noncompact sets
(\cite{bo1}), and  they proved the variational principle under
some supplementary conditions. The notions of the topological
pressure, variational principle and equilibrium states play a
fundamental role in statistical mechanics, ergodic theory and
dynamical systems (see the books \cite{bo3,w2}).

Motivated by Feng and Huang's recent work \cite{fh}, where the authors studied the variational principle between Bowen
topological entropy and measure theoretic entropy for an arbitrary
subset. As a natural generalization of topological entropy, topological pressure is a quantity which
belongs to one of the concepts in the thermodynamic formalism. This study defines measure theoretic pressure for a Borel
probability measure and investigates its variational relation with the  Bowen topological pressure.
The outline of the paper is as follows. The main results, as well as those definitions of the measure
theoretic pressure and topological pressures, are given in Section
2. The proof of the main results and related propositions are given in section 3.

%%%%%%%%%%%%%%%%%%%%%%%%%%%%%%%%%%%%%%%%%%%%%%%%
%%%%%%%%%%%%%%%%%%%%%%%%%%%%%%%%%%%%%%%%%%%%%%%%
\section{Main results}
%\setcounter{section}{1} \setcounter{equation}{0}
%%%%%%%%%%%%%%%%%%%%%%%%%%%%%%%%%%%%%%%%%%%%%%%%
%%%%%%%%%%%%%%%%%%%%%%%%%%%%%%%%%%%%%%%%%%%%%%%%
One of the most fundamental dynamical invariants that associate to a continuous map is the
topological pressure with a potential function. It roughly measures the orbit
complexity of the iterated map on the potential function.
This section first gives these definitions of measure theoretic
pressure for any Borel probability measure, and then
recalls different kinds of definitions of the topological pressure.
The variational relationship of topological pressure and measure
theoretic pressure is stated as the following Theorem A.

We first give some necessary notations as follows. Along with the study of Bowen work in \cite{bo3},
for any $n\in \mathbb{N}$, denote $d_n(x,y)=\max\{d(T^{i}(x),T^{i}(y)):i=0,\cdots,n-1\}$ for any
$x,y\in X$, and $B_n(x,\epsilon)=\{y\in X:d_n(x,y)<\epsilon\}$. A
set $E\subseteq X$ is said to be an $(n,\epsilon)$-separated
subset of $X$ with respect to $T$ if $x,y\in E, x\neq y$, implies
$ d_{n}(x,y)> \epsilon$. Dual definition is as follows. A set $F\subseteq X$ is said to be an
$(n,\epsilon)$-spanning subset of $X$ with respect to $T$ if
$\forall x\in X$, $\exists y\in F$ with $d_{n}(x,y)\leq \epsilon$.
Here, $C(X)$ denotes the Banach space of all continuous functions on
$X$ equipped with the supremum norm $\|\cdot\|$.

\subsection{Measure theoretic pressure}
Let $\mu\in \mathcal M(X)$ and $f\in C(X)$, the \emph{measure
theoretic pressure} of $\mu$ for $T$ (w.r.t.  $f$) is defined by
\[
P_{\mu}(T,f):=\int{P_{\mu}(T,f,x)\,\mathrm{d}{\mu}(x)}
\]
where $P_{\mu}(T,f,x):= \lim\limits_{\epsilon \to
0}\liminf\limits_{n\to
\infty}\Big(\frac{1}{n}\log[e^{f_n(x)}\cdot{\mu}(B_n(x,\epsilon))^{-1}]\Big)$
and $f_n(x):=\sum_{i=0}^{n-1}f(T^ix)$.

For any  ${\mu}\in \mathcal M_T$, using Birkhoff's egodic theorem
(e.g. see \cite{w2}) and Brin and Katok's entropy formula
\cite{brin}, for $\mu-$almost every $x\in X$ we have that
\[
P_{\mu}(T,f,x)=h_{\mu}(T,x)+f^{*}(x)
\]
where $$h_{\mu}(T,x)=\lim_{\epsilon \rightarrow o}\liminf_{n\rightarrow \infty}\frac{-1}{n}\log \mu(B_n(x,\epsilon)). $$
Also $f^{*}\circ T=f^{*},
\int{f^{*}}\,\mathrm{d}\mu=\int f\,\mathrm{d}\mu$ and $\int
h_{\mu}(T,x)\,\mathrm{d}\mu= h_{\mu}(T)$. Particularly, if $\mu\in
\mathcal E_T$ we have that $P_{\mu}(T,f,x)=h_{\mu}(T,x)+f^{*}(x)=h_{\mu}(T)+\int
f\,\mathrm{d}\mu$ for $\mu-$almost every $x\in X$. See \cite{chz,czc,he04,z2,cz09} for more details on the
measure theoretic pressure of invariant measures for a large class of potentials.

In the following subsections, we turn to give definitions of
\emph{upper capacity topological pressure}, \emph{Bowen
topological pressure} and \emph{weighted topological pressure}.
The main idea of those pressures is the extension from that of Feng and Huang's approximations
in \cite{fh}.

\subsection{Upper capacity topological pressure}
 Recall that the \emph{upper capacity topological pressure} of $T$ on a
subset $Z\subseteq X$ with respect to a continuous function $f$ is
given by
\[
P(T,f,Z) =\lim\limits_{\epsilon\rightarrow 0}P(T,f,Z,\epsilon)
\]
where
\begin{eqnarray*}
&&P(T,f,Z,\epsilon)=\limsup\limits_{n\rightarrow\infty}
\frac{1}{n}\log P_{n}(T,f,Z,\epsilon),\\
&&P_{n}(T,f,Z,\epsilon)=\sup \{ \sum\limits_{x\in E} e^{f_{n}(x)}
: E\  \text{is  an}\ (n,\epsilon)\text{-separated subset of}\ Z\}.
\end{eqnarray*}
 This definition is equivalent to the Pesin and  Pitskel's definition which is the standard dynamically defined dimension
 characteristic, see \cite{pes1} for details.

\subsection{Bowen topological pressure}
Let $Z\subseteq X$ be a subset of $X$, which neither has to be
compact nor $T$-invariant.
Fix $\epsilon >0$, we call $\Gamma=\{B_{n_i}(x_i,\epsilon)\}_i$
a \emph{cover of $Z$} if $Z\subseteq\bigcup_i B_{n_i}(x_i,\epsilon)$. For
$\Gamma=\{B_{n_i}(x_i,\epsilon)\}_i$, set
$n(\Gamma)=\min_i\{n_i\}$.

The theory of Carath\'eodory dimension characteristic ensures the
following definitions.

\begin{definition}\label{defPmu*}
Let $f$ be a continuous function and $s\in\mathbb{R}$, put
\begin{eqnarray*}
M(Z,f,s,N,\epsilon)=\inf_\Gamma\sum_i \exp\bigl(-sn_i+\sup_{y\in
B_{n_i}(x_i,\epsilon) } f_{n_i}( y)\bigr),
\end{eqnarray*}
where the infimum is taken over all covers $\Gamma$ of $Z$ with
$n(\Gamma)\geq N$. Then let
%\begin{equation}\label{other2}
\begin{gather*}%\label{other2}
m(Z,f,s,\epsilon)
=\lim_{N\rightarrow\infty}M(Z,f,s,N,\epsilon),  \\
P_{B}(T,f,Z,\epsilon) =\inf \{ s: m(Z,f,s,\epsilon)=0 \}
 =\sup \{ s: m(Z,f,s,\epsilon)=+\infty \}, \\
P_{B}(T,f,Z)=\lim_{\epsilon\rightarrow  0}P_{B}(T,f,K,\epsilon).
\end{gather*}
The term $P_B(T,f,Z)$ is called the \emph{Bowen topological
pressure} of T on the set  $Z$ (w.r.t. $f$).
\end{definition}

The Bowen topological pressure can be defined in an alternative
way, see \cite{ba} or \cite{pes1} for more details.

Suppose $\mathcal U$ is a finite open cover of $X$. Denote the
diameter of the open cover by $ |\mathcal
U|:=\max\left\{\diam(U):U\in \mathcal U\right\} $. For $n\geq 1$
we denote by ${\mathcal W}_n(\mathcal U)$ the collection of
strings $\mathbf{U}=U_1...U_n$ with $U_i\in \mathcal U$. For
$\mathbf{U}\in{\mathcal W}_n(\mathcal U)$ we call the integer
$m(\mathbf{U})=n$ the length of $\mathbf{U}$ and define
$$X(\mathbf{U})=U_1\cap T^{-1}U_2\cap...\cap T^{-(n-1)}U_n=\left\{x\in X:T^{j-1}x\in U_j ~\text{for}~
j=1,...n\right\}.
$$
Let $Z\subseteq X$. We say that $\Lambda\subset\bigcup_{n\geq
1}{\mathcal W}_{n}(\mathcal U)$ covers $Z$ if
$\bigcup_{\mathbf{U}\in\Lambda}X(\mathbf{U})\supset Z$. For
$s\in\mathbb R$, define
\[
M_{N}^{s}(\mathcal
U,f,Z)=\inf_{\Lambda}\sum_{\mathbf{U}\in\Lambda}\exp(-sm(\mathbf{U})+\sup_{y\in
X(\mathbf{U})}f_{m(\mathbf{U})}(y))
\]
where the infimum is taken over all $\Lambda\subset\bigcup_{n\geq
1}{\mathcal W}_{n}(\mathcal U)$ that cover $Z$ and $
\sup\limits_{y\in X(\mathbf{U})}f_{m(\mathbf{U})}(y)=-\infty$ if
$X(\mathbf{U})=\emptyset$. Clearly, $M_N^{s}(\mathcal U,f,\cdot)$
is a finite outer measure on X, and
\[
M_{N}^{s}(\mathcal U,f,Z)=\inf\Big\{M_{N}^{s}(\mathcal
U,f,G),G\supset Z,G\ \text{is open}\Big\}.
\]
Note that $M_{N}^{s}(\mathcal U,f,Z)$ increases as $N$ increases,
define
\[
M^{s}(\mathcal U,f,Z):=\lim_{N\to\infty}M_{N}^{s}(\mathcal U,f,Z)
\]
and
\[
P_B(T,f,\mathcal U,Z):=\inf\left\{s:M^{s}(\mathcal
U,f,Z)=0\right\}=\sup\left\{s:M^{s}(\mathcal
U,f,Z)=+\infty\right\},
\]
set
\[
P_B(T,f,Z):=\sup_{\mathcal U}P_B(T,f,\mathcal U,Z)
\]
From these notations, it is not difficult to prove that $\sup\limits_{\mathcal
U}P_B(T,f,\mathcal U,Z)=\lim\limits_{|\mathcal U|\rightarrow
0}P_B(T,f,\mathcal U,Z)$.

\subsection{Weighted topological pressure}

For any bounded function $g:X\rightarrow \mathbb R$, $f\in C(X)$,
$\epsilon>0$ and $N\in\mathbb{N}$, define
\[
W(g,f,s,N,\epsilon)=\inf\sum_{i}c_i\exp(-sn_{i}+\sup_{y\in
B_{n_{i}}(x_{i},\epsilon)}f_{n_i}{(y)})
\]
where the infimum is taken over all finite or countable families
$\left\{B_{n_{i}}(x_{i},\epsilon),c_i\right\}$ such that
$0<c_i<\infty,{x_i}\in X,{n_i}\geq N $ and
\[
\sum_{i}c_i\chi_{B_i}\geq g,
\]
where $B_i:=B_{n_i}(x_i,\epsilon)$ and $\chi_A$ denotes the
characteristic function on a subset $A\subseteq X$. For
$K\subseteq X$ and $g=\chi_K$ we set
\[
W(K,f,s,N,\epsilon):=W(\chi_K,f,s,N,\epsilon).
\] The quantity
$W(K,f,s ,N ,\epsilon)$ does not decreases as $N$ increases, hence
the following limit exists:
\[
w(K,f,s,\epsilon)=\lim_{N\to\infty}W(K,f,s ,N ,\epsilon).
\]
Clearly, there exists a critical value of the parameter $s$.
Hence, define
\[
P_W(T,f,K,\epsilon)=\inf\left\{s:w(K
,f,s,\epsilon)=0\right\}=\sup\left\{s:w(K,f,s,\epsilon)=\infty\right\}
\]
It is easy to see that the quantity $P_W(T,f,K,\epsilon)$ is
monotone with respect to $\epsilon$, thus the following limit exists:
\[
P_W(T,f,K)=\lim_{\epsilon \to 0}P_W(T,f,K,\epsilon).
\]
The term $P_W(T,f,K)$ is called a \emph{weighted topological
pressure} of $T$ on the set $K$ (with respect to $f$).

Now we collect some properties of the pressures, see \cite{ba} or
\cite{pes1} for proofs.

\begin{proposition}\label{property}
Let $(X,T)$ be a TDS and $f\in C(X)$, then the following
properties hold:
\begin{enumerate}
\item[(i)] For $Z_1\subseteq Z_2$, $\mathcal{P}(T,f,Z_1)\leq
\mathcal{P}(T,f,Z_2)$, where $\mathcal{P}$ is $P,P_B$ or  $P_W$;
\item[(ii)] For $Z=\bigcup\limits_{i=1}^{\infty}Z_i$,
$P_B(T,f,Z)=\sup\limits_{i\ge 1}P_B(T,f,Z_i)$ and $P(T,f,Z)\le
\sup\limits_{i\ge 1}P(T,f,Z_i)$; \item[(iii)] For any $Z\subseteq
X$, $P_B(T,f,Z)\le P(T,f,Z)$. Moreover, we have $P_B(T,f,Z)=
P(T,f,Z)$ if $Z$ is $T-$invariant and compact.
\end{enumerate}
\end{proposition}

The following variational relation between the Bowen topological
pressure and the measure theoretic pressure is the main finding of
this paper. We give the statement first and postpone the proof to
the next section. To formulate our results, we need to introduce an
additional notion. A set in a metric space is said to be
\emph{analytic} if it is a continuous image of the set $\mathcal
N$ of infinite sequences of natural numbers (with its product
topology). It is known that in a Polish space, the analytic
subsets are closed under countable unions and intersections, and
any Borel set is analytic (c.f. \cite[2.2.10]{fed}).
\begin{TheoremA}\label{mainthm}
Let $(X,T)$ be a TDS and $f$ a continuous function on  $X$.
\begin{enumerate}
\item[(1)] If $K\subseteq X$ is non-empty and compact, then
 \[
P_B(T,f,K)=\sup\left\{P_{\mu}(T,f):{\mu}\in{M(X)},{\mu}(K)=1\right\} ;
 \] \item[(2)]If the topological entropy of the system is finite, i.e., $h_{top}(T)<\infty$, and $Z\subseteq X $ is analytic, then
 \[
 P_B(T,f,Z)=\sup\left\{P_B(T,f,K):K\subseteq Z,K\ is\
 compact\right\}.
 \]
\end{enumerate}
\end{TheoremA}

%%%%%%%%%%%%%%%%%%%%%%%%%%%%%%%%%%%%%%%%%%%%%%%%
%%%%%%%%%%%%%%%%%%%%%%%%%%%%%%%%%%%%%%%%%%%%%%%%
\section{Proof of the main result}
\setcounter{equation}{0}
%%%%%%%%%%%%%%%%%%%%%%%%%%%%%%%%%%%%%%%%%%%%%%%%
%%%%%%%%%%%%%%%%%%%%%%%%%%%%%%%%%%%%%%%%%%%%%%%%
In the academic study of a dynamical system $(X,T)$, the well-known variational principle of topological pressure provides
the relationship among pressure, entropy invariants and potential energy from the probabilistic and topological versions.
This section provides a proof of the variational principle for these pressures in Theorem A.
To study the relations of the Bowen topological pressure with the
weighted topological pressure, the following Vitali covering lemma is necessary.
\begin{lemma} \label{cover}
Let $(X,d)$ be a compact metric space and $\mathcal B=\left\{B(x_i,r_i)\right\}_{i\in\mathcal I}$
 be a family of closed (or open) balls in X.
Then there exists a finite or countable subfamily $\mathcal B^{'}=\left\{B(x_i,r_i)\right\}_{i\in\mathcal I^{'}}$
of pairwise disjoint balls in $\mathcal B$ such that
    \[
    \bigcup_{B\in\mathcal B}B\subseteq\bigcup_{i\in\mathcal I^{'}}B(x_i,5r_i)
    \]
 \end{lemma}
\begin{proof}
See \cite[Theorem2.1]{Mat}.
\end{proof}

\begin{proposition} \label{bwpre}
 Let $K\subseteq X$. Then for any $s\in \mathbb{R}$ and $\epsilon,\delta >0$, we have
    \[
    \mathcal M(K,f,s+\delta,N,6\epsilon)\leq W(K,f,s,N,\epsilon)\leq M(K,f,s,N,\epsilon)
    \]
for all sufficiently large $N$, where $\mathcal
M(K,f,s+\delta,N,6\epsilon):=\inf_\Gamma\sum_i
\exp\bigl(-sn_i+f_{n_i}( x_i)\bigr)$ and the infimum is taken over
all covers $\Gamma=\{B_{n_i}(x_i,6\epsilon)\}$ of $K$ with
$n(\Gamma)\geq N$. Consequently, we have $
P_B(T,f,K,6\epsilon)\leq P_W(T,f,K,\epsilon)\leq
P_B(T,f,K,\epsilon)$ and $P_B(T,f,K)=P_W(T,f,K)$.
\end{proposition}
  \begin{proof} We follow Feng and Huang's argument \cite[Proposition 3.2]{fh} to prove this result.
Let $K\subseteq X,\,s\in \mathbb{R},\,\epsilon,\,\delta>0$, taking
$c_i=1$ in the definition of weighted topological pressure, we see
that $W(K,f,s,N,\epsilon)\leq M(K,f,s,N,\epsilon)$ for each  $N\in
\mathbb{N}$.  In the following, we show that $\mathcal
M(K,f,s+\delta,N,6\epsilon)\leq W(K,f,s,N,\epsilon)$ for all
sufficiently large $N$.

Assume that $N \geq 2$
is such that $n^2e^{-n\delta}\leq1$ for $n\geq N$. Let $\left\{B_{n_{i}}(x_{i},\epsilon),c_i\right\}_{i\in \mathcal{I}}$
be a family so that $\mathcal I\subseteq \mathbb{N},x_i\in X,0<c_i<\infty,n_i\geq N $ and
$$\sum_{i}c_i\chi_{B_i}\geq\chi_K$$
where $B_i:=B_{n_i}(x_i,\epsilon)$. We show below that
\begin{eqnarray}\label{mainaim}
\mathcal M(K,f,s+\delta,N,6\epsilon)\leq \sum_{i\in \mathcal
I}c_i\exp(-sn_{i}+\sup_{y\in
B_{n_{i}}(x_{i},\epsilon)}f_{n_i}{(y)})
\end{eqnarray}
which implies $ \mathcal M(K,f,s+\delta,N,6\epsilon)\leq
W(K,f,s,N,\epsilon)$.

Denote ${\mathcal I}_n:=\left\{i\in\mathcal I:n_i=n\right\}$ and
${\mathcal I_{n,k}=\left\{i\in\mathcal I_n:i\leq k\right\}}$ for
$n\geq N$ and $k\in \mathbb N$. Write for brevity
$B_i:=B_{n_i}(x_i,\epsilon)$ and $5B_i:=B_{n_i}(x_i,5\epsilon)$
for $i\in \mathcal I$. Obvisously we may assume $B_i\ne B_j$ for
$i\ne j$. For $t>0$, set
\[
K_{n,t}=\Big\{x\in K:\sum_{i\in{\mathcal I}_n}c_i\chi_{B_i}(x)>t\Big\}~\text{and}~
K_{n,k,t}=\Big\{x\in K:\sum_{i\in{\mathcal I}_{n,k}}c_i\chi_{B_i}(x)>t\Big\}.
\]
We divide the proof of (\ref{mainaim}) into the following three steps.

\emph{Step 1.} This part differs slightly from \cite{fh}, the construction goes through largely verbatim. We write out the details for collecting
some constants.
 For each $n\geq N, k\in\mathbb N$ and $t>0$, there exists a finite
 set $\mathcal J_{n,k,t}\subseteq \mathcal I_{n,k}$ such that the balls $B_i(i\in\mathcal J_{n,k,t})$ are pairwise disjoint,
$K_{n,k,t}\subseteq \bigcup_{i\in\mathcal J_{n,k,t}}5B_i$ and
\[
\sum_{i\in \mathcal J_{n,k,t}}\exp(-sn+\sup_{y\in B_i}
f_n(y))\leq\frac{1}{t}\sum_{i\in\mathcal
I_{n,k}}c_i\exp(-sn+\sup_{y\in B_n(x_i,\epsilon)}f_n(y)).
\]
%where $a:=\inf_{i\in\mathcal I_{n,k}}\Big\{\sup_{y\in B_n(x_i,\epsilon)}f_n(y))\Big\}$.

Since $\mathcal I_{n,k}$ is finite, by approximating $c_i^{'}$s from above, we may assume that each $c_i$ is positive rational,
and then multiplying a common denominator we may assume that each $c_i$ is a positive integer.
Let $m$ be the least integer with $m\geq t$. Denote $\mathcal B=\left\{B_i:i\in\mathcal I_{n,k}\right\}$
and define $u:\mathcal B\to\mathbb{Z}$ by $u(B_i)=c_i$. We define by induction integer-valued functions $v_0,v_1,...,v_m$ on $\mathcal B$ and
subfamilies $\mathcal B_1,...,\mathcal B_m$ of $\mathcal B$ starting with $v_0=u$. Using Lemma \ref{cover} (in which we take the metric $d_n$ instead of $d$)
we find a pairwise disjoint subfamily $\mathcal B_1 $ of $ \mathcal B$ such that $\bigcup_{B\in\mathcal B}B\subseteq\bigcup_{B\in\mathcal B_1}5B$,
 and hence $K_{n,k,t}\subseteq\bigcup_{B\in\mathcal B_1}5B$. Then by repeatedly using Lemma \ref{cover},
we can define inductively for $j=1,...,m$, disjoint subfamilies $\mathcal B_{j}$ of
$\mathcal B$ such that
\[\mathcal B_j\subseteq\left\{B\in\mathcal B:v_{j-1}(B)\geq 1\right\},~~K_{n,k,t}\subseteq\bigcup_{B\in\mathcal B_j}5B
\]
and the function $v_j$ such that
\[
v_j(B)=\left\{\begin{array}{cc}
 v_{j-1}(B)-1&\text{for}~ B\in\mathcal B_j,\\
v_{j-1}(B) &\text{for}~ B\in\mathcal B\setminus\mathcal B_j.
\end{array}
\right.
\]
This is possible since
$K_{n,k,t}\subset\left\{x:\sum\limits_{B\in\mathcal B:B\ni
x}v_j(B)\geq m-j\right\}$ for $j<m$, whence every $x\in K_{n,k,t}$
belongs to some ball $B\in \mathcal B$ with $v_j(B)\geq 1$. Hence,
\begin{eqnarray*}%\sum_{j=1}^m\sharp(\mathcal B_j)\exp(-sn+a)&=&
\sum_{j=1}^m\sum_{B\in\mathcal B_j}\exp(-sn+\sup_{y\in
B}f_n(y))&=&\sum_{j=1}^m\sum_{B\in\mathcal
B_j}(v_{j-1}(B)-v_j(B))\exp(-sn+\sup_{y\in
B}f_n(y))\\
 &\leq &\sum_{B\in\mathcal B}\sum_{j=1}^m(v_{j-1}(B)-v_j(B))\exp(-sn+\sup_{y\in
B}f_n(y))\\
 &\leq &\sum_{B\in \mathcal B}u(B)\exp(-sn+\sup_{y\in
B}f_n(y))\\
 &= & \sum_{i\in\mathcal I_{n,k}}c_i\exp(-sn+\sup_{y\in
B_i}f_n(y)).
\end{eqnarray*}
Choose $j_0\in\left\{1,...,m\right\}$ so that
$\sum\limits_{B\in\mathcal B_{j_0}}\exp(-sn+\sup\limits_{y\in
B}f_n(y))$ is the smallest. Then
\begin{eqnarray*}
\sum\limits_{B\in\mathcal B_{j_0}}\exp\Big(-sn+\sup\limits_{y\in
B}f_n(y)\Big)&\leq&\frac{1}{m}\sum_{i\in\mathcal I_{n,k}}c_i\exp\Big(-sn+\sup_{y\in B_{n}(x_{i},\epsilon)}f_{n}{(y)}\Big)\\
&\leq&\frac{1}{t}\sum_{i\in\mathcal
I_{n,k}}c_i\exp\Big(-sn+\sup_{y\in
B_{n}(x_{i},\epsilon)}f_{n}{(y)}\Big).
\end{eqnarray*}
Hence $\mathcal J_{n,k,t}=\left\{i\in \mathcal I_{n,k}:B_i\in
\mathcal B_{j_{0}}\right\}$ is as desired.

\emph{Step 2.}  For each $n\geq N $ and $ t>0$, we have
\begin{eqnarray}\label{interm}
\mathcal
M(K_{n,t},f,s+\delta,N,6\epsilon)\leq\frac{1}{n^2t}\sum_{i\in\mathcal
I_n}c_i\exp\Big(-sn+\sup_{y\in
B_{n}(x_{i},\epsilon)}f_{n}{(y)}\Big).
\end{eqnarray}

To see this, assume $K_{n,t}\not=\emptyset$; otherwise there is
nothing to prove. It's clear that $K_{n,k,t}\uparrow K_{n,t}$,
$K_{n,k,t}\not=\emptyset$ when $k$ is large enough. Let $\mathcal
J_{n,k,t}$ be the sets constructed in step 1, then $\mathcal
J_{n,k,t}\not=\emptyset$ when $k$ is large enough. Define
$E_{n,k,t}=\left\{x_i:i\in\mathcal J_{n,k,t}\right\}$. Note that
the family of all non-empty compact subsets of $X$ is compact with
respect to the Hausdorff distance (cf. \cite[2.10.21]{fed}). It
follows that there is a subsequence $(k_j)$ of natural numbers and
a non-empty compact set $E_{n,t}\subseteq X$ such that
$E_{n,k_j,t}$ converges to $E_{n,t}$ in the Hausdorff distance as
$j\to\infty$. Since any two points in $E_{n,k,t}$ have a distance
(with respect to $d_n$) not less then $\epsilon$, so do the points
in $E_{n,t}$. Thus $E_{n,t}$ is a finite set and
$\sharp(E_{n,k_j,t})=\sharp (E_{n,t})$ when $j$ is large enough.
Hence
\[
\bigcup_{x\in{E_{n,t}}}B_n(x,5.5\epsilon)\supseteq\bigcup_{x\in
E_{n,k_j,t}}B_n(x,5\epsilon)=\bigcup_{i\in\mathcal
J_{n,k_j,t}}5B_i\supseteq K_{n,k_j,t}
\]
when $j$ is large enough, and thus $\bigcup_{x\in
E_{n,t}}B_n(x,6\epsilon)\supseteq K_{n,t}$. Since
$\sharp(E_{n,k_j,t})=\sharp(E_{n,t})$ when $j$ is large enough,
using the result in step 1 we have
\begin{eqnarray*}
\sum_{x\in E_{n,t}}\exp(-ns+f_n(x))&\leq&\sum_{x\in
E_{n,k_j,t}}\exp(-ns+\sup_{y\in B_n(x,\epsilon)}
f_n(y))\\
&\leq& \frac{1}{t}\sum_{i\in\mathcal I_{n}}c_i\exp(-sn+\sup_{y\in
B_{n}(x_{i},\epsilon)}f_{n}{(y)}).
\end{eqnarray*}
Hence,
\begin{eqnarray*}
\mathcal M(K_{n,t},f,s+\delta,N,6\epsilon)
 &\leq &\sum_{x\in E_{n,t}}\exp\bigr(-n(s+\delta)+f_n(x)\bigr)\\
 &\leq &\frac{1}{e^{n\delta}t}\sum_{i\in\mathcal I_{n}}c_i\exp(-sn+\sup_{y\in B_n(x_i,\epsilon)}f_n(y))\\
 &\leq &\frac{1}{n^2t}\sum_{i\in\mathcal I_{n}}c_i\exp(-sn+\sup_{y\in B_n(x_i,\epsilon)}f_n(y)).
\end{eqnarray*}

\emph{Step 3.} For any $t\in (0,1)$, we have
$$\mathcal M(K,f,s+\delta,N,6\epsilon)\leq\frac{1}{t}\sum_{i\in\mathcal I}c_i\exp(-sn+\sup_{y\in B_{n}(x_{i},\epsilon)}f_{n}{(y)}).$$
As a result, (\ref{mainaim}) holds.

 To see this, fix $t\in (0,1)$. Note that
$\sum_{n=N}^{\infty}n^{-2}<1$ and
$K\subseteq\bigcup_{n=N}^{\infty}K_{n,n^{-2}t}$. By (\ref{interm})
we have
\begin{eqnarray*}
\mathcal M(K,f,s+\delta,N,6\epsilon)
 &\leq &\sum_{n=N}^{\infty}\mathcal M(K_{n,t},f,s+\delta,N,6\epsilon)\\
 &\leq &\sum_{n=N}^{\infty}\frac{1}{t}\sum_{i\in\mathcal I_n}c_i\exp(-sn+\sup_{y\in B_n(x_i,\epsilon)}f_n(y))\\
 &\le &\frac{1}{t}\sum_{i\in\mathcal I}c_i\exp(-sn_{i}+\sup_{y\in B_{n_i}(x_i,\epsilon)}f_{n_i}(y)).
\end{eqnarray*}
To end the proof of this proposition, note that the Bowen
topological pressure does not change if we replace
$\sup\limits_{y\in B_n(x,\epsilon)}f_n(y)$ by any number in the
interval $[\inf\limits_{y\in B_n(x,\epsilon)}f_n(y),
\sup\limits_{y\in B_n(x,\epsilon)}f_n(y)]$ in the definition of
the Bowen topological pressure, see \cite[Corollary 1.2]{ba} or
\cite{pes1} for a proof of this fact.
\end{proof}

The following lemma is an analogue of Feng and Huang's approximation and classic Frostman's lemma, see
\cite[Lemma 3.4]{fh}.

\begin{lemma}\label{measure-cons}
Let K be a nonempty compact subset of $X$ and $f\in C(X)$. Let
$s\in \mathbb R $, $N\in\mathbb N$ and $ \epsilon>0$. Suppose that
$c:=W(K,f,s,N,\epsilon)>0$. Then there is a Borel probility
measure $\mu$ on X such that $\mu(K)=1$ and
\[
\mu(B_n(x,\epsilon))\leq\frac{1}{c}\exp\Big[-ns+\sup_{y\in
B_{n}(x,\epsilon)}f_{n}(y)\Big],~~\forall x\in X,n\geq N.
\]
\end{lemma}
\begin{proof}
Clearly $c<\infty$. We define a function $p$ on the Banach space
$C(X)$
 by
\[
p(g)=\frac{1}{c}W(\chi_{K}\cdot g,f,s,N,\epsilon).
\]

Let $1\in C(X)$ denote the constant function $1(x)\equiv 1$, it is
easy to verify that
\begin{enumerate}
\item[(1)] $p(g+h)\leq p(g)+p(h)\ \text{for any}\ g,h\in C(X)$;
\item[(2)] $p(tg)=tp(g)\ \text{for any}\ t\geq 0\ \text{and}\ g\in
C(X)$; \item [(3)] $p(1)=1,0\leq p(g) \leq \parallel g
\parallel\ \text{for any}\ g\in C(X),\ \text{and}\ p(h)=0\ \text{for}\ h\in
C(X)\ \text{with}\ h\leq 0$.
\end{enumerate}
Applying the Hahn-Banach Theorem, we can extend the linear functional
$t\mapsto tp(1),t\in \mathbb R$, from the subspace of the constant
functions to a linear functional $L:C(X)\rightarrow \mathbb R$
satisfying
\[
L(1)=p(1)=1\ \text{and}\ -p(-g)\leq L(g)\leq p(g)\ \text{for any}\
g\in C(X).
\]
If $g\in C(X)$ with $g\geq 0$, then $p(-g)=0$ and $L(g)\geq 0$.
Hence, combining the fact that $L(1)=1$, we can use the Riesz
representation theorem to find a Borel probability measure $\mu$
on $X$ such that $L(g)=\int g\,\mathrm{d}\mu$ for $g\in C(X)$.

Now we show that $\mu(K)=1$. To see this, for any compact set
$E\subseteq X\setminus K$, by the Uryson lemma there is $g\in
C(X)$ such that $0\leq g(x)\leq 1,g(x)=1\ \text{for}\ x\in E$ and
$g(x)=0$ for $x\in K$. Then $g\cdot\chi_K \equiv 0$ and thus
$p(g)=0$. Hence $\mu(E)\leq L(g)\leq p(g)=0$. This shows $
 \mu(X\setminus K)=0$, i.e., $\mu(K)=1$.

 In the end, we show that
 \[
 \mu(B_n(x,\epsilon))\leq\frac{1}{c}\exp\Big[-ns+\sup_{y\in B_{n}(x,\epsilon)}f_{n}(y)\Big],~\forall x\in X,n\geq N.
 \]
 To see this, for any compact set $E\subset B_n(x,\epsilon)$, by the Urysohn lemma, there exists $g\in C(X)$ such that $0\leq g\leq 1,~g(y)=1$
 for $y\in E$ and $g(y)=0$ for $y\in X\setminus B_n(x,\epsilon)$. This implies that $\mu(E)\leq L(g) \leq p(g)$.
 Since $g\cdot\chi_{K}\leq \chi_{B_n(x,\epsilon)}$ and $n\geq N$, we have
 \[
 W(\chi_{K}\cdot g,f,s,N,\epsilon)\leq \exp\Big[-ns+\sup_{y\in B_{n}(x,\epsilon)}f_{n}(y)\Big]
 \]
 and thus $
 p(g)\leq \frac{1}{c}\exp\Big[-ns+\sup\limits_{y\in B_{n}(x,\epsilon)}f_{n}(y)\Big]$.
 Therefore $\mu(E)\leq \frac{1}{c}\exp\Big[-ns+\sup\limits_{y\in B_{n}(x,\epsilon)}f_{n}(y)\Big]$, it follows that
 \begin{eqnarray*}
 \mu(B_n(x,\epsilon))&=&\sup\left\{\mu(E):E\ \text{is a compact subset of}\ B_{n}(x,\epsilon)\right\}\\
 &\leq& \frac{1}{c}\exp\Big[-ns+\sup_{y\in B_{n}(x,\epsilon)}f_{n}(y)\Big].
 \end{eqnarray*}
 This completes the proof of the lemma.
\end{proof}

Now it's ready to prove the first result in Theorem A.

\begin{proof}[{\bf Proof of Theorem A(i)}]Let $\mu\in \mathcal M(X)$ satisfying $\mu(K)=1$. Write
 \[
 P_{\mu}(T,f,x,\epsilon)=\liminf_{n\to\infty}\frac{1}{n}\log[e^{f_n(x)}.{\mu}(B_n(x,\epsilon))^{-1}]
 \]
 for $x\in X,n\in \mathbb{N}$ and $\epsilon>0$. Since $\frac{1}{n}\log[e^{f_n(x)}\cdot{\mu}(B_n(x,\epsilon))^{-1}]\ge -\|f\|$ for each $n$,
applying Fatou's lemma we have
\[
 \lim_{\epsilon\to 0}\int P_{\mu}(T,f,x,\epsilon)\,\mathrm{d}{\mu}\geq\int P_{\mu}(T,f,x)\,\mathrm{d}{\mu}=P_{\mu}(T,f).
 \]
Thus, to show $P_B(T,f,K)\geq P_{\mu}(T,f)$, it suffices  to prove that$
 P_B(T,f,K)\geq\int P_{\mu}(T,f,x,\epsilon)\,\mathrm{d}\mu$ for each $\epsilon>0$.

Fix $\epsilon>0$ and $l\in\mathbb N$. Denote $u_l=\min\left\{l,\int P_{\mu}(T,f,x,\epsilon)\,\mathrm{d}\mu-\frac{1}{l}\right\}$.
Then there exists a Borel set $A_l\subseteq X$ with $\mu(A_l)>0$ and $N\in\mathbb N$ such that
 \[
 \mu(B_n({x,\epsilon}))\leq \exp\Big(-nu_l+f_n(x)\Big),\forall x\in A_l,n\geq N.
 \]
Now let $\left\{B_{n_i}(x_i,\frac{\epsilon}{2})\right\}$ be a
countable or finite family such that $x_i\in X$, $n_i\geq N$ and
$\bigcup_{i}B_{n_i}(x_i,\frac{\epsilon}{2})\supseteq K\cap A_l$.
We may assume that $B_{n_i}(x_i,\frac{\epsilon}{2})\cap(K\cap
A_l)\not=\emptyset$ for each $i$, and choose $y_i\in
B_{n_i}(x_i,\frac{\epsilon}{2})\cap(K\cap A_l)$, then
\begin{eqnarray*}
\sum_{i}\exp\Big[-n_iu_l+\sup_{y\in B_{n_i}(x_i,\epsilon/2)}f_{n_i}(y)\Big ]
 &\geq &\sum_{i}\exp\Big(-n_iu_l+f_{n_i}(y_i)\Big)\\
 &\geq &\sum_{i}\mu(B_{n_i}(y_i,\epsilon))\\
 &\geq &\sum_{i}\mu(B_{n_i}(x_i,\frac{\epsilon}{2}))\\
 &\geq &\mu(K\cap A_l)=\mu(A_l)>0
\end{eqnarray*}
It follows that
\[
M(K\cap A_l,f,u_l,N,\frac{\epsilon}{2})\geq \mu(A_l)>0.
\]
Therefore $ P_B(T,f,K)\geq P_B(T,f,K\cap A_l)\ge u_l$. Letting
$l\to\infty$, we have  $ P_B(T,f,K)\geq \int
P_{\mu}(T,f,x,\epsilon)\,\mathrm{d}\mu$. Hence
\[P_B(T,f,K)\geq P_{\mu}(T,f).
\]

We next show that
\begin{eqnarray}\label{UB}
P_B(T,f,K)\leq \sup\Big\{P_{\mu}(T,f):\mu\in M(X),\mu(K)=1\Big\}.
\end{eqnarray}
We can assume that $P_B(T,f,K)\neq -\infty$, otherwise we have
nothing to prove. By Proposition \ref{bwpre} we have
$P_B(T,f,K)=P_W(T,f,K)$.  Fix a small number $\beta>0$. Let
$s=P_B(T,f,K)-\beta$. Since
\[
\lim_{\epsilon\to
0}\liminf_{n\to\infty}\frac{1}{n}[f_n(x)-\sup_{y\in
B_{n}(x,\epsilon)}f_n(y)]=0
\]
for all $x\in X$, we have that
\[
\liminf_{n\to\infty}\frac{1}{n}[f_n(x)-\sup_{y\in
B_{n}(x,\epsilon)}f_n(y)]>-\beta,~\forall x\in X
\]
for all sufficiently small $\epsilon>0$. Take such an $\epsilon>0$
and a $N\in \mathbb N$ such that $c:=W(K,f,s,N,\epsilon)>0$. By
Lemma \ref{measure-cons}, there exists $\mu\in M(X)$ with
$\mu(K)=1$ such that
$\mu(B_n(x,\epsilon))\leq\frac{1}{c}\exp\Big[-ns+\sup\limits_{y\in
B_{n}(x,\epsilon)}f_{n}(y)\Big]$ for any $x\in X$ and $n\geq N$.
Therefore $$ P_{\mu}(T,f,x)\ge P_{\mu}(T,f,x,\epsilon)\geq
s+\liminf_{n\to\infty}\frac{1}{n}[f_n(x)-\sup_{y\in
B_{n}(x,\epsilon)}f_n(y)]\ge P_B(T,f,K)-2\beta$$ for all $x\in X$.
Hence,
\[P_{\mu}(T,f)=\int P_{\mu}(T,f,x)\,\mathrm{d}\mu\geq P_B(T,f,K)-2\beta.
\]
Consequently, (\ref{UB}) is obtained immediately.
\end{proof}

Next we turn to prove the second result in Theorem A. We will first prove this result in the case of that $X$ is \emph{zero-dimensional}, and then
prove it in general.
Now we prove a useful lemma first.

\begin{lemma}\label{cpt}
Assume that $\mathcal U$ is a closed-open partition of X. Let $N\in \mathbb N$ and $f\in C(X)$.
\begin{enumerate}
\item[(i)] If $\ E_{i+1}\supseteq E_{i}$ and $\bigcup_{i}E_i=E$, then $M_{N}^{s}(\mathcal U,f,E)=\lim\limits_{i \to\infty}M_{N}^{s}(\mathcal U,f,E_{i})$;
\item[(ii)] Assume $Z\subset X$ is analytic. Then $M_{N}^{s}(\mathcal U,f,Z)=\sup\left\{M_{N}^{s}(\mathcal U,f,K):K\subset Z,K\ \text{is compact}\right\}$.
\end{enumerate}
\end{lemma}
\begin{proof}
We first show that (i) implies (ii). Assume that (i) holds. Let
$Z$ be analytic, i.e., there exists a continuous surjective map
$\Phi:\mathcal N \to Z$. Let $\Gamma_{n_1,n_2,\cdots,n_p} $ be the
set of $(m_1,m_2,\cdots)\in\mathcal N$ such that $m_1\leq
n_1,m_2\leq n_2,\cdots,m_p\leq n_p$ and let $Z_{n_1,\cdots,n_p}$
be the image of $\Gamma_{n_1,\cdots,n_p}$ under $\Phi$. Let
$\{\epsilon_p\}_{p\ge 1}$ be a sequence of positive numbers. Due
to (i), we can pick a sequence $\{n_p\}_{p\ge 1}$ of positive
integers recursively so that $M_{N}^{s}(\mathcal U,f,Z_{n_1})\geq
M_{N}^{s}(\mathcal U,f,Z)-{\epsilon}_1$ and
\[
M_{N}^{s}(\mathcal U,f,Z_{n_1,\cdots,n_p})\geq M_{N}^{s}(\mathcal
U,f,Z_{n_1,\cdots,n_{p-1}})-\epsilon_p,~~p=2,3,\cdots
\]
Hence,
\[
M_{N}^{s}(\mathcal U,f,Z_{n_1,\cdots,n_p})\geq M_{N}^{s}(\mathcal
U,f,Z)-\sum_{i=1}^{\infty}\epsilon_{i}, ~~\forall p\in \mathbb N.
\]
Let
\[K=\bigcap_{p=1}^{\infty}\overline{Z_{n_1,\cdots,n_p}}.
\]
Since $\Phi$ is continuous, we can show that $
\bigcap_{p=1}^{\infty}\overline{Z_{n_1,\cdots,n_p}}=\bigcap_{p=1}^{\infty}{Z_{n_1,\cdots,n_p}}$ by applying Cantor's diagonal argument.
 Hence $K$ is a compact subset of $Z$. If $\Lambda \subset\bigcup_{j
\geq N}{\mathcal W}_{j}(\mathcal U)$ is a cover of $K$ (of course it is an open cover), then it is a cover of $\overline{Z_{n_1,\cdots,n_p}}$
when $p$ is large enough, which implies
\[
\sum_{ \mathbf{U}\in\Lambda}\exp\Big(-sm(\mathbf{U})+\sup_{y\in
X(U)}f_{m(\mathbf{U})}(y)\Big)\geq \lim_{p\to\infty}
M_{N}^{s}(\mathcal U,f,Z_{n_1,\cdots,n_p})\geq M_{N}^{s}(\mathcal
U,f,Z)-\sum_{i=1}^{\infty}\epsilon_i.
\]
Hence $M_{N}^{s}(\mathcal U,f, K)\geq M_{N}^{s}(\mathcal U,f,Z)-\sum_{i=1}^{\infty}\epsilon_i$.
Since $\sum_{i=1}^{\infty}\epsilon_i$ can be chosen arbitrarily small, we have proved (ii).

Now we turn to prove (i). Note that any two non-empty elements in
${\mathcal W}_{n}(\mathcal U)$ are disjoint, and each element in
${\mathcal W}_{n+1}(\mathcal U)$ is a subset of some element in
${\mathcal W}_{n}(\mathcal U)$. We call this the net property of
$\{{\mathcal W}_{n}(\mathcal U)\}_{n\ge 1}$. Let $E_{i}\uparrow E$
be given. Let $\{\delta_{i}\}_{i\ge 1}$ be a sequence of positive
numbers to be specified later and for each $i$, choose a cover
${\Lambda}_{i}\subset\bigcup_{j\geq N}{\mathcal W}_{j}(\mathcal
U)$ of $E_{i}$ such that
\begin{eqnarray}\label{one}
\sum_{\mathbf{U}\in{\Lambda}_{i}}\exp\Big(-sm(\mathbf{U})+\sup_{y\in
X(\mathbf{U})}f_{m(\mathbf{U})}(y)\Big)\leq M_{N}^{s}(\mathcal
U,f,E_{i})+{\delta_{i}}.
 \end{eqnarray}
By the net property of $\{{\mathcal W}_{n}(\mathcal U)\}_{n\ge
1}$, we may assume these elements in ${\Lambda}_{i}$ are disjoint
for each $i$.

For any $x\in E$, choose
$\mathbf{U}_{x}\in\cup_{i=1}^{\infty}{\Lambda}_{i}$ such that
$X(\mathbf{U}_{x})$ containing $x$ and $m(\mathbf{U}_x)$ is the
smallest. By the net property of $\{{\mathcal W}_{n}(\mathcal
U)\}_{n\ge 1}$, the collection $\left\{\mathbf{U}_{x}:x\in
E\right\}$ consists of countable many disjoints elements. Relabel
these elements as $\mathbf{U}_{i}'s$. Clearly
$E\subset\bigcup_{i}X(\mathbf{U}_{i})$.

We now choose an integer $k$. Let ${\mathcal A}_{1}$ denote the
collection of those ${\mathbf{U}_{i}}'s$ that are taken from
${\Lambda}_1$. They cover a certain subset $Q_1$ of  $E_k$. The
same subset is covered by a certain sub-collection of
${\Lambda}_k$, denoted as ${\Lambda}_{k,1}$, since
${\Lambda}_{k,1}$
 also covers the smaller set $Q_{1}\cap E_{1}$, by (\ref{one})
 \begin{eqnarray} \label{two}
 \sum_{\mathbf{U}\in{\mathcal A}_{1}}\exp\Big(-sm(\mathbf{U})+\sup_{y\in X(\mathbf{U})}f_{m(\mathbf{U})}(y)\Big)\leq
 \sum_{\mathbf{U}\in{\Lambda}_{k,1}}\exp\Big(-sm(\mathbf{U})+\sup_{y\in X(\mathbf{U})}f_{m(\mathbf{U})}(y)\Big)+{\delta}_{1}.
 \end{eqnarray}
To see this, assume that (\ref{two}) is false. Then by (\ref{one}),
\[
 \sum_{\mathbf{U}\in({\Lambda}_{1}\setminus {\mathcal A}_{1})\bigcup{\Lambda}_{k,1}}\exp\Big(-sm(\mathbf{U})+\sup_{y\in X(\mathbf{U})}f_{m(\mathbf{U})}(y)\Big)<M_{N}^{s}(\mathcal U,f,E_{1})
 \]
 which contradicts the fact that $({\Lambda}_{1}\setminus {\mathcal A}_{1})\bigcup{\Lambda}_{k,1}\subset \bigcup_{j\geq N}{\mathcal W}_{j}(\mathcal U)$
is a open cover of $E_1$. Next we use ${\mathcal A}_{2}$ to denote
the collection of those ${\mathbf{U}_i}'s$ that are taken from
${\Lambda}_{2}$ but not from ${\Lambda}_{1}$. Define
${\Lambda}_{k,2}$ similarly.
 As above,we find
\begin{eqnarray} \label{three}
 \sum_{\mathbf{U}\in {\mathcal A}_{2}}\exp\Big(-sm(\mathbf{U})+\sup_{y\in X(\mathbf{U})}f_{m(\mathbf{U})}(y)\Big)\leq
\sum_{\mathbf{U}\in
\Lambda_{k,2}}\exp\Big(-sm(\mathbf{U})+\sup_{y\in
X(\mathbf{U})}f_{m(\mathbf{U})}(y)\Big)+{\delta}_{2}.
 \end{eqnarray}
We repeat the argument until all coverings ${\Lambda}_{n},n\leq
k$, have been considered. Note that
$\bigcup_{\mathbf{U}\in\Lambda_{k,i}}\mathbf{U}\subseteq
\bigcup_{\mathbf{U}\in {\mathcal A}_{i}}\mathbf{U}$ for $i\leq k$.
For different $i,i'\leq k$, the elements in $\Lambda_{k,i}$ are
disjoint from those in $\Lambda_{k,i'}$. The $k$ inequalities
(\ref{two}),(\ref{three}), $\cdots$, are added which yield
 \begin{eqnarray*}
&&\sum_{\mathbf{U}\in \bigcup_{n=1}^{k}{\mathcal
A}_n}\exp\Big(-sm(\mathbf{U})+\sup_{y\in
X(\mathbf{U})}f_{m(\mathbf{U})}(y)\Big)\\
 &&\leq \sum_{\mathbf{U}\in\bigcup_{n=1}^{k}\Lambda_{k,n}}\exp\Big(-sm(\mathbf{U})+\sup_{y\in X(\mathbf{U})}f_{m(\mathbf{U})}(y)\Big)+\sum_{n=1}^{k}\delta_{n}\\
 &&\leq  M_{N}^{s}(\mathcal U,f,E_k)+\sum_{n=1}^{k}\delta_{n}+\delta_{k}.
\end{eqnarray*}
Let $k\to \infty$, we have
\[\sum_{i}\exp\Big(-sm(\mathbf{U}_i)+\sup_{y\in X(\mathbf{U}_i)}f_{m(\mathbf{U}_i)}(y)\Big)\leq \lim_{k\to\infty}M_{N}^{s}(\mathcal U,f,E_k)+\sum_{n=1}^{\infty}\delta_{n}
\]
Since $\sum_{n=1}^{\infty}\delta_{n}$ can be chosen arbitrarily small, it follows that
$$M_{N}^{s}(\mathcal U,f,E)\leq \lim_{k\to\infty}M_{N}^{s}(\mathcal U,f,E_k).$$
Clearly, the opposite inequality is trivial, thus (i) is proven.
\end{proof}

\begin{theorem}\label{zero}
Let $(X,T)$ be a $TDS$. Assume that $X$ is zero-dimensional, i.e., for any $\delta >0$, $X$ has a closed-open partition with diameter less then $\delta$.
Then, for any analytic set $Z\subseteq X$,
\[
P_B(T,f,Z)=\sup\left\{P_B(T,f,K):K\subseteq Z,K\ \text{is  compact}\right\}.
\]
\end{theorem}
\begin{proof}
Let $Z$ be an analytic subset of $X$ with $P_B(T,f,Z)\neq
-\infty$, otherwise there is nothing to prove. Let $s<P_B(T,f,Z)$.
Since $P_B(T,f,Z)=\sup\limits_{\mathcal U}P_B(T,f,\mathcal
U,Z)=\lim\limits_{|\mathcal U|\rightarrow 0}P_B(T,f,\mathcal
U,Z)$, there exists a closed-open partition $\mathcal U$ so that
$P_B(T,\mathcal U,f,Z)>s$ and thus $M^{s}(\mathcal U,f,Z)=\infty$.
Hence $M_{N}^{s}(\mathcal U,f,Z)>0$ for some $N\in \mathbb N$. By
Lemma \ref{cpt}, we can find a compact set $K\subseteq Z$ such
that $M_{N}^{s}(\mathcal U,f,K)>0$. This implies $P_B(T,f,K)\geq
P_B(T,\mathcal U,f,K)\geq s$. This is the result that we need.
\end{proof}

\begin{proposition}\label{uctop}
Let $(X,T)$ be a $TDS$ with $h_{top}(T)<\infty $ and $f\in C(X)$,
then there exists a factor $\pi:(Y,S)\rightarrow(X,T)$ such that
$(Y,S)$ is zero-dimensional and
\[
\sup_{x\in X}P(S,f\circ\pi,\pi^{-1}(x))\leq \| f \|.
\]
\end{proposition}
\begin{proof} Assume that $(X,T)$ is a $TDS$ with
$h_{top}(T)<\infty$. By  Lemma 3.13 in \cite{fh},  there exists a
factor $\pi:(Y,S)\rightarrow(X,T)$ such that $(Y,S)$ is
zero-dimensional and
\[
\sup_{x\in X}P(S,0,\pi^{-1}(x))=0.
\]
This immediately implies that $\sup_{x\in
X}P(S,f\circ\pi,\pi^{-1}(x))\leq \| f \|$.
\end{proof}

\begin{proposition}\label{bprelation}
If $\pi:(Y,S)\rightarrow(X,T)$ is a factor map and $f$ is a
continuous function  on $X$, then for $E\subset Y$ we have
\[
P_B(T,f,\pi(E))\leq P_B(S,f\circ\pi,E)\leq
P_B(T,f,\pi(E))+\sup_{x\in X}P(S,f\circ\pi,\pi^{-1}(x)).
\]
\end{proposition}
\begin{proof} See \cite[Theorem 2.1]{lcz} for the proof of
the second inequality. It is left to prove the first inequality.
Fix $\epsilon >0$. By the uniform continuity of the map $\pi$,
there exists $\delta >0$ such that
\[
d(x,y)<\delta\Longrightarrow d(\pi(x),\pi(y))<\epsilon.
\]
Fix a positive integer $N$, consider a cover of $E$ with Bowen
balls $\{B_{n_{i}}(x_{i},\epsilon)\}$, where $n_i\geq N$ for each
$i$. Then it is easy to see that
$\{B_{n_{i}}(\pi(x_{i}),\epsilon)\}$ is a cover of $\pi(E)$, and
 $M(\pi(E),f,s,N,\epsilon)\leq M(E,f\circ\pi,s,N,\delta)$.
This implies that
\[
M(\pi(E),f,s,\epsilon)\leq M(E,f\circ\pi,s,\delta).
\]
Hence, $ P_B(T,f,\pi(E),\epsilon)\leq P_B(S,f\circ\pi,E,\delta)$.
Since $\epsilon\rightarrow 0$ implies $\delta\rightarrow0$, let
$\epsilon\rightarrow 0$ we have
\[
P_B(T,f,\pi(E))\leq P_B(S,f\circ\pi,E).
\]
This completes the proof of the theorem.
\end{proof}

\begin{proposition}\label{equal}
Let $(X,T)$ be a $TDS$ with $h_{top}(T)<\infty $, then there
exists a factor $\pi:(Y,S)\rightarrow(X,T)$ such that $(Y,S)$ is
zero-dimensional and
\[
P_B(T,f,\pi(E))=P_B(S,f\circ\pi,E),~~\forall E\subseteq Y.
\]
\end{proposition}
\begin{proof}By Proposition \ref{uctop}, there exists a factor $\pi:(Y,S)\rightarrow(X,T)$ such that $(Y,S)$ is
zero-dimensional and $\sup_{x\in X}P(S,f\circ\pi,\pi^{-1}(x))\leq
\| f \|$.

Since for any $c\in \mathbb{R}$ and $f\in C(X)$, we have $
P_B(T,f+c,Z)=P_B(T,f,Z)+c~~\text{and}~~P(S,f\circ\pi+c,\pi^{-1}(x))=P(S,f\circ\pi,\pi^{-1}(x))+c$.
Applying Proposition \ref{bprelation} for the function $f-\|f\|$,
we have
\[
P_B(T,f,\pi(E))\leq P_B(S,f\circ\pi,E)\leq
P_B(T,f,\pi(E))+\sup_{x\in X}P(S,f\circ\pi,\pi^{-1}(x))-\parallel
f
\parallel.
\]
This implies that
\[
P_B(T,f,\pi(E))=P_B(S,f\circ\pi,E).
\]
\end{proof}

Now we turn to prove the second result in Theorem A.

\begin{proof}[{\bf Proof of Theorem A(ii)}] By Proposition \ref{equal}, there exists a factor map $\pi:(Y,S)\rightarrow (X,T)$ such that $(Y,S)$ is
zero-dimensional and $P_B(T,f,\pi(E))=P_B(S,f\circ\pi,E)$ for any
$f\in C(X)$ and $E\subseteq Y$.

 Let $Z$
be an analytic subset of $X$. Then $\pi^{-1}(Z)$ is also an
analytic subset of $Y$. Using Proposition \ref{zero}, we have
\begin{eqnarray*}
P_B(T,f,Z)
 &= &P_B(S,f\circ\pi,\pi^{-1}(Z))\\
 &= &\sup \left\{P_B(S,f\circ\pi,E):E\subseteq \pi^{-1}(Z),E\ \text{is compact}\right\}\\
 &= &\sup \left\{P_B(T,f,\pi(E)):E\subseteq \pi^{-1}(Z),E\ \text{is compact}\right\}\\
 &\leq&\sup \left\{P_B(T,f,K):K\subseteq Z,K\ \text{is compact}\right\}
 \end{eqnarray*}
 By Proposition \ref{property}, the reverse inequality is trivial.
 Hence,
\[
P_B(T,f,Z)=\sup\left\{P_B(T,f,K):K\subseteq Z,K\ \text{is
compact}\right\}.
\]
\end{proof}

%%%%%%%%%%%%%%%%%%%%%%%%%%%%%%%%%%%%%%%%%%%%%%%%
%%%%%%%%%%%%%%%%%%%%%%%%%%%%%%%%%%%%%%%%%%%%%%%%
 \noindent {\bf Acknowledgements.}
This work is
partially supported by NSFC (11001191) and Ph.D. Programs Foundation
of Ministry of Education of China (20103201120001).

\end{document}